\documentclass[11pt]{article}
\usepackage{graphicx}
\usepackage{amsmath,amsfonts,amssymb,latexsym,epsfig}
\usepackage{mathrsfs}
\usepackage{verbatim}
\usepackage{latexsym}
\usepackage{amsthm}
\usepackage{amssymb}
\usepackage{graphics}
\usepackage{amsbsy}
\usepackage{enumerate}
\usepackage{times}

\renewcommand{\nabla}{D}

\newtheorem{theorem}{Theorem}[section]
\newtheorem{lemma}[theorem]{Lemma}

\newtheorem{proposition}[theorem]{Proposition}

\newtheorem{definition}[theorem]{Definition}
\newcommand{\dist}{\mathrm{dist}}
\newcommand{\veps}{\varepsilon}
\newcommand{\R}{\mathbb{R}}

\newcommand{\dd}{\mathrm{d}}

\newcommand{\meas}{\mathrm{meas}}
\newcommand{\Ib}{I_{\eps}}
\newcommand{\Iob}{J_{\eps}}
\newcommand{\xt}{x_{\rm c}}

\newcommand{\supp}{\mathrm{supp}}

\numberwithin{equation}{section}
\newcommand{\Ho}{H^{1}_{\pm}}

\newcommand{\bbR}{\mathbb{R}}

\newcommand{\eps}{\varepsilon}

\newcommand{\xs}{x^{\star}}

\begin{document}
\setlength{\baselineskip}{10pt}
\title{$\Gamma-$Limit for Transition Paths\\
of Maximal Probability}
\author{
F.J. Pinski
\footnote{E-mail address: frank.pinski@uc.edu} \\
        Physics Department\\
        University of Cincinnati\\
	PO Box 210011\\
	Cincinnati OH 45221, USA\\
and\\
        A.M. Stuart and F. Theil
\footnote{E-mail address: \{a.m.stuart,f.theil\}@warwick.ac.uk.} \\
        Mathematics Institute \\
        Warwick University \\
        Coventry CV4 7AL, UK
                    }
\maketitle

\begin{abstract}
Chemical reactions can be modelled via diffusion
processes conditioned to make a transition between
specified molecular configurations representing
the state of the system before and after the chemical
reaction. In particular the model of Brownian
dynamics -- gradient flow subject to additive noise --
is frequently used. If the chemical reaction is specified
to take place on a given time interval, then
the most likely path taken by the system is a
minimizer of the Onsager-Machlup functional.
The $\Gamma-$limit of this functional is determined
in the case where the
temperature is small and the transition time
scales as the inverse temperature.
\end{abstract}

\section{Introduction}
\label{sec:intro}

In this paper we study the problem of
determining the most likely paths
taken by a system undergoing a
chemical reaction. We employ the model
known as {\em Brownian dynamics} \cite{AT87}:
the atomic positions are assumed to
be governed by a gradient flow in an interaction
potential, subject to additive thermal noise.
The resulting stochastic differential
equation is conditioned to make a transition between
two different atomic confirgurations representing
the state of the system before and after the chemical
reaction \cite{BDGC}.
If the chemical reaction is specified
to take place on a given time interval, then
the most likely path taken by the system is a
minimizer of the Onsager-Machlup functional
\cite{durrbach,ikeda-watanabe-89}. In \cite{ps09}
numerical computations are presented which study minimizers
of this functional for a variety of model problems
in low dimension, together with some higher dimensional
problems from physics and chemistry such as vacancy
diffusion and the Lennard-Jones cluster of $38$ atoms.
The minimizers exhibit a number of interesting
effects, including multiple pathways between
configurations, together with transition paths
which concentrate at saddle points of the potential,
and not at minima. The computations in \cite{ps09}
were performed at low temperature over fixed long intervals.
This paper is concerned with determining the $\Gamma-$limit
(see \cite{Braid02,dal_maso})
of the Onsager-Machlup functional in the case where the
temperature is small and the transition time
scales as the inverse temperature.
Minimizers of the Onsager-Machlup functional have
received considerable attention in the chemistry
literature (see \cite{OE99} and \cite{ps09} and
the references therein). One of the motivations for
the work in \cite{ps09}
is to show that there are certain artefacts
in this minimization procedure which could be
construed as unphysical. In this paper we provide
a mathematical theory to explain the
computational observations in \cite{ps09}.
The $\Gamma-$limit which we exhibit may also
be of independent interest.
The $\Gamma-$limit can be
optimized via knowledge of the critical points
of the potential, the trace of the Hessian
of the potential at the critical points, and
certain heteroclinic orbits connecting them.
Notably these heteroclinic orbits are in a
Hamiltonian system; heteroclinic orbits in the
forward or backward
gradient flow found from our model at zero
temperature are solutions of this
Hamiltonian problem but not necessarily vice versa.
Regarding the role of the Hessian in the
$\Gamma-$limit, it is pertinent to mention
the paper \cite{voss08} in
which second derivative of the potential plays
a role in a large deviations principle for
SDEs.

In section \ref{sec:setup} we provide a precise
mathematical description of the conditioned SDE
which forms our mathematical model. We provide
an informal derivation of the Onsager-Machlup
functional whose minimizers determine most likely
transition paths and we explain the sense in which
this informal argument can be made rigorous.
In section \ref{sec:gamma} we compute the
$\Gamma-$limit
of the Onsager-Machlup functional building on
related analyses in \cite{ks89,sb90}.
The paper concludes, in section \ref{sec:num},
with an informal characteriztion of the
$\Gamma-$limit, together with numerical experiments
which illustrate this characterization.

\section{Set-Up}
\label{sec:setup}

Consider the following conditioned SDE for
$x \in C([0,T];\bbR^{N})$ making a transition
between two states $x^-$ and $x^+$ in time $T$:
\begin{align}
\label{eq:SDE}
\begin{split}
&\dd x=-\nabla V(x)\dd t+\sqrt{2\eps}\,\dd W,\\
&x(0)=x^-\quad{\mbox {and}}\quad x(T)=x^+.
\end{split}
\end{align}
Here $V:\bbR^{N} \to \bbR$ is the potential, $W$
is a standard Brownian motion in $\bbR^{N}$
and $\eps \ll 1.$
In many applications in physics/chemistry
$N=nd$, $d\in \{1,2,3\}$ being the physical
dimension and $n$ the number of particles. Then $\eps$ is
the (non-dimensional) temperature and
$x \in C([0,T];\bbR^{N})$
denotes
the configurational path of the $n$ atoms
making a transition between $x^{\pm}$.
Throughout we assume that $x^{\pm}$ are chosen as
critical points of $V$. We will be particularly interested
in choosing these critical points to be minima,
so that the problem \eqref{eq:SDE} describes
a chemical reaction; but, in the numerical computations,
we will also choose saddle points
in order to illustrate certain mathematical phenomena.

Escape from local minima occurs on timescales which
are exponentially long in $\eps^{-1}$, whilst transitions
themselves occur on scales which are logarithmic in
$\eps^{-1}$.  We focus interest on an intermediate
timescale between these two regimes, which is $\eps^{-1}$.
This timescale is hence long enough to
capture a single transition, but not long enough to
capture the typical waiting time in a potential minimum.
In particular, transition paths calculated in this
scaling which pass through an
intermediate minimum $x^0$ of $V$ will not necessarily
exhibit the ``typical'' behaviour
that would be exhibited if
the time to make the transition from $x^-$ to $x^+$
via $x^0$ were left free. The behaviour of the most
likely single transitions in this regime is captured
by minimizers of the Freidlin-Wentzell
action \cite{FreidWentz84,ERV02}.
Numerical methods to capture multiple
such transitions are the subject of active study,
and the reader should consult \cite{ERV05b}, and the
references therein, for details. From an applied
perspective, our work serves to highlight the
potential pitfalls of using the Onsager Machlup
approach to compute transition paths, and this point
is discussed in detail in \cite{ps09}. As mentioned
above, the work in this paper
gives a mathematical explanation
for the numerical results observed in \cite{ps09}.
The results may also be of independent interest
from the point of view of the calculus of variations.
In particular the results show how hetereoclinic orbits
in a certain Hamiltonian flow form the building block
for construction of the $\Gamma-$limit. This fact
may also be of interest in applications where attention
has focussed on heteroclinic orbits in the forward
or backward gradient flow found from \eqref{eq:SDE}
with $\epsilon=0$; these form particular instances of
heteroclinic orbits in the Hamiltonian flow, but
non-gradient connections are also possible, as we will
demonstrate.

To enforce the scaling of interest
we choose $T=\eps^{-1}$ and
rescale time as $t=\eps^{-1}s$, to obtain
\begin{align}
\label{eq:rescale}
\begin{split}
&\dd x=-\frac{1}{\varepsilon}\nabla V(x)ds+\sqrt{2}\,
\dd W,\\
&x(0)=x^-\quad{\mbox {and}}\quad x(1)=x^+.
\end{split}
\end{align}
In this scaling we are studying
transitions on a unit time-interval,
in which the systematic motion of the molecules is large
and the thermal noise is of order $1$.

The probability measure $\pi$ governing the stochastic boundary
value problem \eqref{eq:rescale} has density
with respect to the Brownian bridge measure
$\pi_0$ arising in the case $V \equiv 0.$ The density
relating the two measures is found from the Girsanov formula,
together with an integration by parts (use of It\^{o} formula)
and use of the boundary conditions on the path \cite{RV05,HSV07}, and is given by
\begin{equation}
\label{eq:rd}
\frac{d\pi}{d\pi_0}(x) \propto \exp\left(-\frac{1}{2\eps^2}\int_0^1 G(x;\eps)\dd t \right)
\end{equation}
where the {\em path potential} $G$ is (with
$|\cdot|$ denoting the Euclidean norm)
\begin{equation}
G(x;\eps)=\frac{1}{2}|\nabla V(x)|^2-\eps\Delta V(x).
\label{eq:G}
\end{equation}
Functions are infinite dimensional and there
is no Lebesgue measure in infinite dimensions.
Nonetheless, it is instructive to
think heuristically of the Brownian bridge obtained
from \eqref{eq:rescale} with $V \equiv 0$
as having a probability density
(with respect to Lebesgue measure) of the form \cite{ChHa}
\begin{equation}
\label{eq:pdfbb}
\exp\left(-\frac{1}{4}\int_0^1 \left|\frac{\dd x}{\dd s}\right|^2 \dd s \right).
\end{equation}
(This is the formal limit obtained from the probability
density function for a discretized Brownian bridge).
Let $\Ho\bigl((0,1)\bigr)$ denote
the subset of $H^1\bigl((0,1);\bbR^{N}\bigr)$
comprised of functions satisfying $x(0)=x^-$
and $x(1)=x^+$.  Then, combining \eqref{eq:rd}
and \eqref{eq:pdfbb}, we may think of the probability
density for $\pi$ as being proportional to
$\exp\bigl(-\frac{1}{2\eps}\Ib(x)\bigr)$ where the
Onsager-Machlup functional
$\Ib:\Ho\bigl((0,1)\bigr)\to \bbR$
is defined by
\begin{equation}
\label{eq:pdf}
\Ib(x):=\int_0^1 \Bigl(\frac{\eps}{2}
\left|\frac{\dd x}{\dd s}\right|^2 +
\frac{1}{\eps} G(x;\eps)\Bigr)\dd s.
\end{equation}

This intuitive definition of $\Ib(x),$ via
the logarithm
of the pathspace probability density function,
suggests that minimizers of $\Ib$ are related to
paths of maximal proability. This can be
formulated precisely as follows
\cite{durrbach,ikeda-watanabe-89}.
Let $B^{\delta}(z)$
denote a ball of radius $\delta$ in $C([0,1];\bbR^{N})$
centred on $z \in \Ho\bigl((0,1)\bigr).$ Then, for any
$z_1,z_2 \in \Ho\bigl((0,1)\bigr)$,
$$\lim_{\delta \to 0} \frac{\pi\bigl(B^{\delta}(z_1)\bigr)}{\pi\bigl(B^{\delta}(z_2)\bigr)}=\exp
\Bigl(\frac{1}{2\epsilon}\bigl(\Ib(z_2)-\Ib(z_1)\bigr)\Bigr).$$
Thus, for small ball radius $\delta$, the logarithm of the
ratio of the probabilities of the two balls is equal to
the difference in $\Ib$ evaluated at the ball centres.
For this reason we are interested in minimizers of $\Ib.$

The following notation will be useful. We define
\begin{align}
\label{eq:IJ}
\begin{split}
\Iob(x)&=\int_0^1 \Bigl(\frac{\eps}{2}
\left|\frac{\dd x}{\dd s}\right|^2 +\frac{1}{2\eps}|\nabla V(x)|^2\Bigr)\dd s,\\
J(x)&=\int_{-\infty}^{\infty} \Bigl(\frac{1}{2}
\left|\frac{\dd x}{\dd s}\right|^2 +\frac{1}{2}|\nabla V(x)|^2\Bigr)\dd s.
\end{split}
\end{align}
Then
\begin{equation}
\label{eq:Ib}
\Ib(x)=\Iob(x)-\int_0^1 \Delta V\bigl(x(s)\bigr)\dd s.
\end{equation}

Our goal in what follows is to demonstrate that
the $\Gamma-$limit of $\Ib$ is finite only when
evaluated on the set of BV functions supported on the
critical points of $V$. And, furthermore, that on
this set the value of the $\Gamma-$limit is
determined by (suitably rescaled) minima of $J$,
subject to end-point conditions, together with the
integral of $\Delta V(x).$ Thus we conclude this
section with some observations concerning
minima of $J$.

The Euler-Lagrange equations for the functional $J$ are
\begin{equation}
\frac{d^2x}{\dd s^2}-D^2 V(x)\nabla V(x)=0,
\label{eq:el}
\end{equation}
where $D^2V$ denotes the Hessian of $V$.
These equations are Hamiltonian and conserve the energy
$$E=\frac{1}{2}\left|\frac{\dd x}{\dd s}\right|^2 -\frac{1}{2}|\nabla V(x)|^2.$$
Thus heteroclinic orbits connecting critical points
of $V$ via the equation \eqref{eq:el}, for which
$E$ is necessarily zero, satisfy
\begin{equation}
\label{eq:par}
\left|\frac{\dd x}{\dd s}\right|=|\nabla V(x)|.
\end{equation}
Hence, if there is a heteroclinic orbit connecting
critical points of $V$ under either of the forward or
backward gradient flows
\begin{equation}
\label{eq:grad}
\frac{\dd x}{\dd s}=\pm \nabla V(x)
\end{equation}
then this will also determine a heteroclinic orbit in
the Hamiltonian system. However, the converse is
not necessarily the case: there are heteroclinic orbits
in the Hamiltomnian flow which are not heteroclinic
orbits in the gradient dynamics.

We say that a potential $V:\bbR^{N} \to \bbR$ is
{\em admissible} if $V \in C^3(\R^{N},\R)$ and
\begin{enumerate}
\item the set of critical points
$$\mathcal E = \{x \in \R^N \; | \; \nabla V(x)=0 \}$$
is finite;
\item the Hessian $D^2 V(x)$ has no
zero eigenvalues for every $x \in \mathcal E$;
\item the weak coercivity condition
\begin{align} \label{weak-coercivity}
\exists R>0 \text{ such that } \inf_{|x|>R} |\nabla V(x)| >0
\end{align}
is satisfied.
\end{enumerate}
Admissibility implies, in particular,
that all critical points for the
gradient flows \eqref{eq:grad} are hyperbolic.

For each pair $x^-,x^+ \in \mathcal E$ the set of
{\em transition paths} is defined as
$$X(x^-,x^+)=\left\{y \in BV(\R) \; \left| \; \lim_{t \to \pm \infty} y(t) = x^\pm \text{ and } \dot y \in L^2(\R) \right.\right\}.$$

\begin{lemma} \label{lem:lower}
For any $x \in X(x^-,x^+)$
\begin{equation}
\label{eq:lb2}
J(x) \ge \bigl|V\bigl(x^+\bigr)-V\bigl(x^{-}\bigr)\bigr|.
\end{equation}
Furthermore, if the infimum of $J$ over $X(x^-,x^+)$
is attained at $\xs \in X(x^-,x^+)$,
then $\xs \in C^2(\R)$ satisfies the Euler-Lagrange
equation \eqref{eq:el} and has the property
\begin{equation}
\label{eq:one}
J(\xs)=\int_{-\infty}^{\infty}\bigl| \nabla V\bigl(\xs(s)\bigr)
\bigr|^2\,\dd s.
\end{equation}
Assume that the potential is admissible.
If either $x^{-}$ or $x^{+}$ is a local minimum
or maximum and there
exists a heterclinic orbit $\xs$ connecting $x^-$ and $x^+$
under the Hamiltonian dynamics \eqref{eq:el}, then
$\xs$ solves \eqref{eq:grad} and
\begin{equation}
\label{eq:two}
J(\xs) = \bigl|V\bigl(x^+\bigr)-V\bigl(x^{-}\bigr)\bigr|.
\end{equation}
In particular $\xs$ is a minimizer of $J$ which attains the
lower bound the lower bound
\eqref{eq:lb2}.
\end{lemma}

\begin{proof}
Clearly
\begin{align*}
J(x) &\ge \int_{-\infty}^{\infty} \Bigl|\frac{\dd x}{\dd s}\Bigr|
\Bigl|\nabla V(x)\Bigr|\,\dd s\\
& \ge \int_{-\infty}^{\infty}  \Bigl|\frac{\dd}{\dd s}
\bigl(V(x)\bigr)\Bigr|\,\dd s\\
& \ge \Bigl|\int_{-\infty}^{\infty}  \frac{\dd}{\dd s}
\bigl(V(x)\bigr)\,\dd s\Bigr|.
\end{align*}
Integrating and using the end point conditions gives
the first result.
Standard regularity results imply that minimizers $x$ of $J$
are $C^2(\R)$ and satisfy the Euler-Lagrange
equation~\eqref{eq:el}.
Equation~\eqref{eq:one} then follows from \eqref{eq:par}.

To prove the final result we first note that
the linearization of equation \eqref{eq:el}
at a critical point $\xt\in \mathcal E$ is given by
\begin{equation}
\frac{\dd^2y}{\dd s^2}-D^2 V(\xt)^2 y=0,
\label{eq:el2}
\end{equation}
whilst the forward gradient flow, given
by \eqref{eq:grad} with a minus sign,
has linearization
\begin{equation}
\frac{\dd y}{\dd s}+D^2 V(\xt) y=0.
\label{eq:el3}
\end{equation}
If $V$ is admissible, and $\xt$ is a critical point of $V$, then
$(0,\xt)$ is hyperbolic for the Hamiltonian
flow and, from \eqref{eq:el2}, has stable and unstable
manifolds both of dimension $N.$

We assume next that $x^+$ is a local minimum of $V$.
Let $\mathcal M$ be the 
$N$-dimensional manifold defined by 
$$\mathcal M =\{ (x, -D V(x)) \; | \; x \in \R^N \} \subset \R^{2N}.$$ 
Differentiating equation \eqref{eq:grad}
with respect to $t$ shows that $\mathcal M$ is invariant under
the flow of equation \eqref{eq:el}. Since $D^2 V(x^+)$ is positive 
definite there exists $\veps>0$ such that all solutions $y(t)$ of 
equation \eqref{eq:el} with the properties $|y(0) - x^+|<\veps$
and $\dot y(0))= - DV(y(0))$  converge to $x^+$ as $t$
tends to infinity. 
Define next the stable manifold
$$ \mathcal M(x^+) = \left\{ (x,v) \in \R^{2N}\; \left| \;
\lim_{t\to \infty} y(t) = x^+\right.\right\},$$
where $y$ solves~\eqref{eq:el} and satisfies the initial condition
$y(0) = x$, $\dot y(0) = v$.
Since the dimension of $\mathcal M(x^+)$ is $N$ the sets 
$\mathcal M(x^+) \cap \{|x-x^+|<\veps\}$ and $\mathcal M \cap 
\{|x-x_+|<\veps\}$ coincide. Thanks to the invariance of $\mathcal M$ 
it follows that $\mathcal M(x^+) \subset \mathcal M$.
As a result one obtains that
\begin{align*}
J(\xs) = -\int_{-\infty}^{\infty} D V(\xs(t)) \dot \xs=
V(x^-) - V(x^+) \geq 0,
\end{align*}
where the last inequality holds because \eqref{eq:grad} implies that
$V$ is a decreasing Lyapunov function along the
trajectory $\xs$. Thus, the claim has been proven under the assumtion
that $x^+$ is a local minimum. The remaining three cases ($x^+$ is a
local maximum, $x^-$ is a local minimum/maximum) can be dealt with in
an analogous fashion.
\end{proof}

The result shows that only saddle-saddle connections can give rise to
heteroclinic orbits in the Hamiltonian system \eqref{eq:el}
which cannot be found by in one of the gradient flows \eqref{eq:grad}.
As a closing remark in this section we observe
that the existence of heteroclinic orbit in \eqref{eq:el}
is a generic property of non-degenerate systems.
For the Hamiltonian flow the effective dimension of
the space is $2N-1,$ because it is constrained to a level
set of the Hamiltonian. Since the stable and unstable
manifolds for the Hamiltonian system are both of dimension $N$
(cf eq. \eqref{eq:el2}) we expect that, generically,
there will be a $1$ dimensional manifold connecting
any pair of equilibria under the Hamiltonian flow.

\section{The $\Gamma-$Limit}
\label{sec:gamma}

We now determine the $\Gamma-$limit
for the Onsager-Machlup function
\eqref{eq:pdf}. A thorough introduction to $\Gamma$-convergence can
be found in \cite{dal_maso} and \cite{Braid02}.
Problems closely related to ours are studied
using $\Gamma$-convergence in \cite{ks89,sb90}.

We start with some basic
definitions and propositions which
serve to explain the form of the
$\Gamma-$limit. The key result is
contained in Proposition \ref{transitionstates}
which demonstrates that every transition is achieved
via a {\em finite} collection of intermediate
transitions.
The limit Theorem \ref{t:gamlim}
is then stated and proved, building on
a number of lemmas which follow it.

Recall the definition of admissible potential $V$, as
well as that of a transition path.

\begin{proposition}\label{gamlim.prop}
Let $V$ be admissible and define for each pair
$x^-,x^+ \in \mathcal E$ the function $\Phi(x^-,x^+)$ by
$$ \Phi(x^-,x^+) = \inf\left\{J(y) \; \left| \; y \in X(x^-,x^+)\right.\right\}.$$
Then there exists $c>0$ such that $\Phi(x^-,x^+)>c$ for all
pairs $x^{\pm} \in \mathcal E$ with $x^-\neq x^+.$
\end{proposition}

We can establish a direct representation of the transition energy $\Phi$
which avoids the usage of infima.

\begin{proposition} \label{transitionstates}
Let $V$ be admissible and $x^\pm \in \mathcal E$ be two critical points of $V$. Then there exists a finite
sequence $\{x_i\}_{i=0}^{k} \in \mathcal E$ such that
$x_0=x^-$, $x_k= x^+$ and
$$ \Phi(x^-,x^+) = \sum_{i=1}^k \min \left\{J(y) \; \left| \; y \in X(x_{i-1},
x_i) \right.\right\},$$
\end{proposition}

A straight-forward refinement of the analysis shows that
the sequence $x_i$ is injective.

\begin{definition} Let $X$ be a Banach-space, $\veps>0$ a parameter and
$I_\veps:X\to\R$ a family of functionals. The functional $I_0:X \to \R\cup \{+\infty\}$ is the $\Gamma-$limit of $I_\veps$ as
$\veps\to 0$ if for all $x\in X$ and all sequences $x_\veps \in X$ which converge weak-* to $x$ as $\veps$ tends to
0 the liminf-inequality
\begin{align*}
\liminf_{\veps \to 0} I_\veps(x_\veps) \geq I_0(x)
\end{align*}
holds, and for all $z \in X$,
there exists a recovery sequence $z_\veps \in X$ which converges weak-* to $z$ and satisfies
\begin{align*}
\limsup_{\veps \to 0} I_\veps(z_\veps) \leq I_0(z).
\end{align*}
\end{definition}

\begin{theorem}\label{t:gamlim}
Let $V$ be an admissible potential.  Then the
$\Gamma$-limit of the functional $I_\veps$ as $\veps$ tends to 0
is
$$ I_0(x) = \left\{\begin{array}{rl} \sum\limits_{\tau \in \mathcal D}
\Phi(x^-(\tau), x^+(\tau)) -
\int_0^1 \Delta V(x(s))
\, \dd s & \text{ if } x\in BV([0,1])\\
& \text{ and }
x(s) \in \mathcal E \text{ a.e. } s\in [0,1],\\[1em]
+ \infty & \text{ else,}
 \end{array} \right.$$
where $\mathcal D(x)$ is the set of discontinuity points of $x$
and $x^\pm(\tau)$ are the left and right-sided limits
of $x$ at $\tau$.
\end{theorem}

Proposition \ref{transitionstates}
shows that the infimum $\Phi$ can be written
as a finite sum of minima. These minima are obtained
from evaluation of \eqref{eq:one}, where $\xs$ solves
\eqref{eq:el} subject to $\xs(t) \to x_{i-1}/x_{i}$ as
$t \to \pm \infty.$ Furthermore, if either $x{i-1}$ or $x_i$ is a local
extremum of $V$, then $\xs$ is also
a heteroclinic orbit in one of the gradient flows
\eqref{eq:grad}, then the minimum is given by \eqref{eq:two}.
Thus Theorem \ref{t:gamlim} shows that the
$\Gamma-$limit of $I_{\veps}$ can be computed through knowledge of the
critical points of $V$ and the set of (Hamiltonian or
gradient) heteroclinic orbits connecting them.

We start the proof of Theorem \ref{t:gamlim}
by proving a lemma which delivers a lower bound
for the amount of energy needed to reach one of the stationary points $x \in \mathcal E$ when
starting nearby. We will often write $x^\pm$ as a shorthand
for $\{x^-,x^+\}$.
\begin{lemma} \label{bv.lem}
Let $V$ be an admissible potential. There exists numbers $C, \veps_0>0$
such that for and all $\veps \in [0,\veps_0]$, $x^{\pm} \in \mathcal E$
and all all paths $y \in X(x^-,x^+)$ with  $\dist(y(0),x^{\pm}) \geq \veps$ the
estimate
\begin{align*}
J(y)\geq \frac{1}{2C} \veps^2
\end{align*}
holds.
\end{lemma}
\begin{proof}
We first show that there exist numbers
$\delta, C>0$ such that, for all $x,y \in \R^N$ with $x\in
\mathcal{E}$ and with the property
$|y-x| <\delta$, the estimates
\begin{align}\label{DVbd.eq}
\frac{1}{C} |y-x|^2 \leq |DV(y)|^2 \leq C |y-x|^2
\end{align}
Indeed, thanks to the smoothness of $V$ and nondegeneracy
condition $\det(D^2 V(x))\neq 0$ for $x \in \mathcal{E}$, for
sufficiently small $\veps$ and large $\tilde C$
\begin{align*}
&\left|DV(y)- D^2 V(x)\cdot(y-x)\right| \leq \tilde C |y-x|^2,\\
&|z| \leq \tilde C \left|D^2 V(x)\cdot z\right| \qquad \forall z \in \R^N
\end{align*}
hold. Thus
\begin{align*}
|\nabla V(y)| \geq
|D^2 V(x)\cdot(y-x)|-\tilde C|y-x|^2\\
\geq\tfrac{1}{\tilde C}|y-x|-\tilde C|y-x|^2
\geq  \tfrac{1}{2\tilde C}|y-x|
\end{align*}
if $|y-x|\leq \min\{\delta,\frac{1}{2\tilde C^2}\}$. This shows that
the first inequality in (\ref{DVbd.eq}) holds. The second inequality
is obtained in a similar way.

Let $y \in X(x^-,x^+)$ and recall that by Sobolev's imbedding theorem
$y$ is continuous.
For each $\veps \in \left(0, \min\left\{\delta, \tfrac{1}{2 C^2}\right\}
\right)$
we define
\begin{equation} \label{tveps}
 t^+(\veps)=\min\left\{s \in \R \; \left| \; |y(t')-x^+|< \veps
\quad\text{ for all } t'>s\right.\right\},
\end{equation}
and $t^-$ analogous. Then we obtain
\begin{align*}
J(y) \geq& \int_{t^+(\veps)}^{\infty} |\nabla V(y(s))| \, |\dot y(s)|
\,\dd s \geq \frac{1}{2C} \int_{t^+(\veps)}^\infty(x^+-y)\cdot \dot y(s) \,
\dd s\\
=& \frac{1}{4C} |y(t^+(\veps))-x^+|^2
= \frac{1}{4C} \veps^2.
\end{align*}
The claim follows since we can derive the same estimate with $t^-$ instead
of $t^+$.
\end{proof}
\begin{proof}[Proof of Proposition~\ref{gamlim.prop}]
Let $r>0$ be the separation of the
stationary points of $V$, i.e
$$r = \min\{ |x-x'| \; | \; x,x' \in \mathcal E, \; x \neq x'\},$$ and
let $y_l \in X(x^-,x^+)$ be such that $\lim_{l \to \infty} J(y_l)
=\Phi(x^-,x^+)$. There exists a sequence $t_l \in \R$ such that that
$\dist(y_l(t_l),\mathcal E) \geq r/2$ and thus Lemma~\ref{bv.lem}
applied to the translated sequence $y_l(\cdot -t_l)$ implies that
$\inf_{l} J(y_l)\geq c$, with $c = \frac{1}{2C}
\min\{r/2,\veps_0\}^2.$
\end{proof}

The proof of Proposition \ref{transitionstates},
showing the existence of minimizing connecting orbits,
is established with the direct
method of the calculus of variations, with the
aid of the following lemma.

\begin{lemma} \label{exist}
Let $x^\pm \in \mathcal E$.
If there exists a minimizing sequence $\{y_l\}_{l \in \mathbb{N}}
\in X(x^-,x^+)$ such that the density $\rho_l(t) =
\frac{1}{2}\left(|\dot y_l(t)|^2 + |D V(y_l(t))|^2\right)$
is tight in $L^1(\R)$, then there exits a minimizer
$y^\pm \in X(x^-,x^+)$ such that $J(y^\pm) = \Phi(x^-,x^+).$
\end{lemma}
\begin{proof}
First we show that the boundary conditions are not lost during the
passage to the limit.  The tightness of $\rho_l$ implies that for each
$\veps>0$ there exists $\tau^{\pm}(\veps) \in \bbR$ with
$\tau^-(\veps)<\tau^+(\veps)$ and the property that, for all $l$,
\begin{align}
\label{tightness.eq}
\limsup_{l\to \infty}\frac{1}{2}\int_{\R\setminus (\tau^-,\tau^+)}\left(|\dot y_l|^2 + |D V(y_l(s))|^2\right)\, \dd s \leq \frac{1}{4C} \veps^2,
\end{align}
where $C$ is the constant in Lemma~\ref{bv.lem}.
Let
$$t_l^+(\veps)=\min\left\{s \in \R \; \left| \; |y_l(t')-x^+|< \veps
\quad\text{ for all } t'>s\right.\right\},$$
and $t_l^-$ analogously.
Without loss of generality we assume that $t^-(\veps)\leq 0 \leq
t_l^+(\veps)$ and $|y_l(0)-x^{\pm}| \geq \veps$.
Then Lemma~\ref{bv.lem} implies that
\begin{align*}
 &\limsup_{l\to \infty}\frac{1}{2}\int_{t^+_l}^\infty \left(|\dot y_l|^2 + |D V(y_l(s))|^2\right)\, \dd s \geq \frac{1}{4C} \veps^2.
\end{align*}
Together with inequality~(\ref{tightness.eq})
we obtain that $\limsup_{l\to \infty} t_l^+(\veps) \leq \tau^+(\veps)$.
Another application of this argument with $t^+$ replaced by $t^-$
delivers the desired result concerning the boundary
conditions:
for each $\veps>0$ there exists $-\infty<\tau^-(\veps)<\tau^+(\veps)< \infty$
such that
\begin{align} \label{clamping}
\left\{ \begin{array}{l}
\limsup\limits_{l \to \infty} \sup\limits_{s \geq \tau^+}|y_l(s)-x^+| \leq \veps,\\
\limsup\limits_{l \to \infty} \sup\limits_{s \leq \tau^-}|y_l(s)-x^-| \leq \veps.
\end{array}\right.
\end{align}
From the bound on $\dot y_l$ in
$L^2\bigl(\tau^-(\veps),\tau^+(\veps)\bigr)$ we
deduce that the length of the vector $y_l(0)\in \R^N$ is bounded.
Since $J(y_l)$ is bounded the sequence $\dot y_l$ is bounded
in $L^2(\R)$.
We extract a subsequence $y_l$ (not relabeled) such that
$\dot y_l$
converges weakly to $\eta \in L^2(\R)$ and $y_l(0)$
converges to $z \in \R^d$.

We define the limiting path $y^\pm$ by
$$y^\pm(t) = z + \int_{0}^t \eta(s)\, \dd s$$
Equation (\ref{clamping}) implies that $\lim_{s \to \pm \infty}
y^\pm(s) = x^\pm$, i.e. the boundary conditions are satisfied.

Next we demonstrate that $y^\pm$ is a minimizer, i.e.
$$ \lim_{l \to \infty} J(y_l) = J(y^\pm).$$
Note that $J$ is weakly lower semicontinuous: the first
term is weakly lower semicontinuous because it is convex
whilst the second term is weakly continuous. To see the
weak continuity note that
\begin{align*}
&\int_{\R} |\nabla V(y^\pm(s))|\,\dd s = \lim_{R \to \infty} \int_{-R}^R
|\nabla V(y^\pm(s))|\,\dd s\\
= &\lim_{R \to \infty} \lim_{l \to \infty} \int_{-R}^R |\nabla V(y_l(s))|\,\dd s
= \lim_{l \to \infty} \int_{-\infty}^\infty |\nabla V(y_l(s))|\,\dd s.
\end{align*}
The last equation is again due to the tightness of $J(y_l)$, the second but last
equation holds because of Sobolev's embedding theorem.
The weak lower semiconinuity of $J$, coupled with the
fact that $\{y_l\}$ is a minimizing sequence, shows that
$J(y^\pm)=\lim_{l \to \infty} J(y_l).$

Finally we show that $\dot y^\pm \in L^1(\R)$, which entails the claim
$y^\pm \in X(x^-,x^+)$.
Define the functional
$$G_t(y) = \frac{1}{2}\int_t^\infty\left(|\dot{y}(s)|^2+ |DV(y(s)|^2\right)
\, \dd s,
$$
and the function $ g(t) = G_t(y^\pm)$.
Note that $(y^\pm(s))_{s\geq t}$ minimizes
$G_t(y)$ for every $t,$
subject to the boundary condition $y(t) = y^\pm(t).$
Testing $G_t(\cdot)$ with a function which
is affine on $[t,t+1]$ and assumes the value $x^+$ for all $s \geq t+1$
together with estimate (\ref{DVbd.eq}) delivers the bound
\begin{align} \label{dp.eq}
g(t) \leq \frac{1}{2} |y^\pm(t)-x^+|^2 \int_0^1\left(1+ C s^2\right)\,
\dd s =  \frac{1}{2}\left(1+\frac{C}{3}\right)|y^\pm(t)-x^+|^2.
\end{align}

Recall that $\lim_{t \to \infty}y^\pm(t) = x^+$ and choose $t_0>0$ such that
$|y(t)-x^+|<\veps_0$ for all $t>t_0$, where $\veps_0$ is defined in
Lemma~\ref{bv.lem}.
The function $g$ satisfies for all $t>t_0$ the differential inequality
\begin{align*}
-\frac{\dd g}{\dd t}(t) \geq \frac{1}{2}|DV(y^\pm(t))|^2\geq
\frac{1}{2 C}|y^\pm(t)-x^+|^2 \geq \mu g(t),
\end{align*}
with $\mu = \frac{3}{(3+C)C}$.
The last inquality is (\ref{dp.eq}), the penultimate inequlity is due
to~\eqref{DVbd.eq}. Gronwall's inequality implies the exponential decay
bound
$g(t) \leq \veps_o e^{-\mu(t-t_0)}.$
Since $|\dot y^\pm(t)|^2 \leq - 2 \frac{\dd g}{\dd t}(t)$ we have shown that
$\dot y \in L^1([t_0,\infty))$. The same argument also shows that
$y^\pm \in L^1((-\infty,t_0])$ and thus $y^\pm \in L^1(\R)$.
\end{proof}

The proof of Proposition~\ref{transitionstates} also
relies on Lions' concentration compactness lemma
which we state here for completeness.

\begin{lemma}\cite{lions84}
\label{l:lions}
Let $(\rho_l)_{l \geq 1}$ be a sequence in $L^1(\R)$ satisfying
$$ \rho_l \geq 0 \text{ in } \R \text{ and } \lim_{l \to \infty}\int_{\R} \rho_l(t) \, \dd t = \lambda,$$
where $\lambda>0$ is fixed. Then, there exists a subsequence (not relabeled)
satisfying one of the following three possibilities:
\begin{enumerate}
\item (compactness) there exists $t_l \in \R$ such that $\rho_l(\cdot - t_l)$ is tight, i.e.
    $$ \forall \veps>0, \exists R< \infty \text{ such that }
    \int_{t_l-R}^{t_l+R} \rho_l(s)\, \dd s \geq \lambda -\veps;$$
\item (vanishing) $$\lim_{l \to \infty} \sup_{t \in \R} \int_{t-R}^{t+R} \rho_l(s)\, \dd s =0\text{ for all } R>0;$$
\item (splitting) there exists $0<\alpha<\lambda$ such that for all $\veps>0$
there exists $l_0\geq 1$ and $\rho^1_l,\rho^2_l \in L^1_\geq(\R)$
such that for all $l\geq l_0$
$$\left\{\begin{array}{rl}
\| \rho_{l}^1 + \rho_l^2-\rho_l\|_{L^1} + \bigl|\|\rho_l^1\|_{L^1}-
\alpha\bigr|
+\bigl|\| \rho_l^2\|_{L^1} +\alpha-\lambda\bigr| \leq \veps,\\
\lim_{l\to \infty} \dist(\supp(\rho_l^1), \supp(\rho_l^2))= \infty.
\end{array} \right.$$
\end{enumerate}
\end{lemma}

\begin{proof}[Proof of Proposition~\ref{transitionstates}]
Let $y_l \in X(x^-,x^+)$ be a minimizing sequence of $J$ and
define $\rho_l(t)$ as in Lemma \ref{exist}.
We now use Lemma \ref{l:lions} in the following way: we
show that vanishing cannot occur, and that splitting can
only occur a finite number of times.

To rule out vanishing we define, for real number
$\veps$ satisfying
$$\veps\in\Bigl(0,\frac{1}{2}\min\{ |x-z| \; \big| \; x,z \in \mathcal E, \; x
\neq z\}\Bigr),$$
the starting and ending time of the final transition:
\begin{align*}
t_l^+ & = \sup
\left\{ t \in \R \; | \; \inf_{s\leq t} |y_l(s) -x^+| \geq \veps
\right\},\\
t_l^- &= \sup
\left\{ t\leq t_l^+ \; | \; \dist(y_l(t), \mathcal E\setminus
\{x^+\}) \leq \veps\right\}.
\end{align*}
By construction $t_l^- \leq t_l^+$ and Jensen's inequality implies
that
\begin{align*}
\int_{t_l^-}^{t_l^+} \rho_l(s) \, \dd s \geq& \frac{1}{2} \int_{t_l^-}^{t_l^+}
\left( \frac{\dist(x^+,\mathcal E\setminus \{x^+\})-
2\veps}{t_l^+-t_l-}\right)^2\, \dd s\\
= & \frac{\left(\dist(x^+,\mathcal E\setminus \{x^+\})-
2\veps\right)^2}{2(t_l^+-t_l^-)}.
\end{align*}
Since $\limsup_{l\to \infty} \int_{\R} \rho_l(s) \, \dd s < \infty$
this implies that $\liminf_{l \to \infty} (t_l^+-t_l^-) >0$.
Furthermore, since $V$ admissible
there exists a number $c>0$ such that $\inf_{s
  \in [t_l^-,t_l^+]} \rho_l(s)\geq c$ and thus
$$\limsup_{l \to \infty} (t_l^+-t_l^-) \leq \frac{1}{c}
\limsup_{l \to \infty} \int_{t_l^-}^{t_l^+}
\rho_l(s) \, \dd s < \infty.$$
This pair of inequalities shows that $\rho_l$ does not vanish.

If the sequence $\rho_l$ splits, then there exist sequences $a_l \leq b_l$
such that $[a_l,b_l] \cap \supp(\rho_l^1 + \rho_l^2) = \emptyset$ and
\begin{align} \label{separ}
\lim_{l \to \infty}(a_l-b_l) = \infty,\\
\label{twobumps}
\liminf_{l\to \infty} \int_{-\infty}^{a_l} \rho_l(t) \, \dd t>0
\text{ and }
\liminf_{l\to \infty} \int_{b_l}^\infty \rho_l(t) \, \dd t>0.
\end{align}
We claim now that there exists a third sequence $t_l \in [a_l,b_l]$ and
a stationary point $x \in \mathcal E$ such that
$\lim_{l\to  \infty} y_l(t_l)=x$.
This is because
\begin{align*}
\limsup_{l\to \infty}\int_{a_l}^{b_l} \rho_l(t)\, \dd t< \infty,
\end{align*}
which together with (\ref{separ}) and the weak
coercivity assumption
(\ref{weak-coercivity})
delivers that
$$\lim_{l \to \infty} \inf\left\{ \dist(y_l(t),\mathcal E) \; | \; t \in
[a_l,b_l]\right\}= 0.$$

We define next the sequences $y_l^1 \in X(x^-,x)$ and $y_l^2 \in
X(x,x^+)$ as follows:
\begin{align*}
y_l^1(t) = \left\{\begin{array}{rl}
y_l(t) & \text{ if } t\leq t_l,\\
(t_l-t+1)y_l(t_t)+(t-t_l)x & \text{ if } t_l<t<t_l+1,\\
x & \text{ if } t\geq t_l+1,
\end{array} \right.\\
y_l^2(t) = \left\{\begin{array}{rl}
x & \text{ if } t\leq t_l-1,\\
(t-t_l+1)y_l(t_l)+ (t_l-t)x & \text{ if } t_l-1<t<t_l,\\
y_l(t) & \text{ if } t\geq t_l.
\end{array} \right.
\end{align*}
Clearly $y_l^1+y_l^2-x \in X(x^-,x^+)$ and
the convergence $\lim_{l\to \infty} y_l(t_l) = x$ implies that
\begin{align} \label{ensplit}
\lim_{l \to \infty} J(y_l^1+y_l^2-x) =
\lim_{l \to \infty} \left(J(y_l^1)+J(y_l^2)\right) = \lim_{l\to \infty}
J(y_l).
\end{align}
Moreover, equation~(\ref{twobumps}) implies
\begin{align} \label{quantized}
\liminf_{l \to \infty} J(y_l^1)>0 \text{ and } \liminf_{l \to \infty} J(y_l^2)>0.
\end{align}
We will show now that $x \not\in \{x^\pm\}$. Indeed, if $x = x^-$, then
$y_l^2 \in X(x^-,x^+)$ and by~(\ref{ensplit})
$$ \liminf_{l \to \infty} J(y^2_l) \leq \lim_{l\to \infty} J(y_l)
-\liminf_{l\to \infty} J(y_l^1).$$
Equation~(\ref{quantized}) delivers a contradiction to
the assumption that $y_l$ is a minimizing sequence.
An analogous argument can
be constructed for the case $x = x^+$
and thus $x \in \mathcal E \setminus \{x^\pm\}$.

The construction yields the following equation
\begin{align*} \nonumber
\Phi(x^-,x^+)\geq& \lim_{l\to \infty}
J(y_l^1+y_l^2-x) \geq  \liminf_{l\to \infty} J(y_l^1) + \liminf_{l\to \infty}
J(y_l^2)\\
 \geq& \Phi(x^-,x) +\Phi(x,x^+).
\end{align*}

Progressing inductively we obtain
a sequence $x_0,x_1,\ldots , x_k \in \mathcal E$ such that
$x_0=x^-$, $x_k=x^+$, $x_{i-1} \neq x_{i}$ for all $i=1\ldots k$ and
\begin{align*}
\Phi(x^-,x^+) \geq \sum_{i=1}^k \Phi(x_{i-1},x_{i}).
\end{align*}
Proposition~\ref{gamlim.prop} implies that $k\leq \Phi(x^-,x^+)/c$,
placing a finite bound on the number of splittings possible.  The
minimizing sequences associated with $X(x_{i-1},x_i)$ are eventually
tight and so achieve their contribution to the infimum at a minimizer
of $J$.
\end{proof}

\begin{proof}[Proof of Theorem~\ref{t:gamlim}]
Now the representation formula for the $\Gamma$-limit will be verified.
Let $x \in BV([0,1],\R^d)$ be a limit path and $x_\veps \in H^1(0,1)$
such that $x_\veps$ converges to $x$ weak-* in $BV([0,1])$.  First we
show that the Laplacian can be treated separately.  Since weak-*
convergence entails boundedness in $BV$ there is a constant $C>0$ such
that $\limsup_{\veps\to 0}\|x_\veps\|_{L^\infty}\leq C$. Furthermore,
thanks to Helly's theorem we can select a subsequence which converges
pointwise for almost every $s \in [0,1]$ and therefore by dominated
convergence
\begin{align*}
\lim_{\veps \to 0} \int_0^1 \Delta V(x_\veps(s))\, \dd s=
\int_0^1 \Delta V(x(s))\, \dd s.
\end{align*}
Note that only continuity of $\Delta V$ is required for this step.
This implies that
\begin{equation}
\liminf_{\veps \to 0} I_\veps(x_\veps) = \liminf_{\veps\to 0}
J_\veps(x_\veps) - \int_0^1 \Delta V(x(s))\, \dd s
\label{eq:sep}
\end{equation}
with $J_\veps$ as defined in \eqref{eq:IJ}.

Next we consider the case where the set of times $s$ where the
limit function $x(s)
\in \mathcal E$ does not have full measure, so that
$I_0(x)=\infty$.  Define
$$M=\meas\left(\left\{t\in [0,1] \; | \; x(t) \not \in \mathcal E\right\}
\right).$$
There exists
$\delta>0$ such that
$$\liminf_{\veps>0} \meas\left(\left\{ t \in [0,1] \; | \;
\dist(x(t),\mathcal E)> \delta\right\}\right) >M/2.$$
Since $V$ satisfies the weak coercivity condition
the infimum of $\left|\nabla V(x)\right|$
exceeds some $\lambda>0$, on this set with measure exceeding $M/2.$
Thus
\begin{align*}
\liminf_{\veps\to 0} J_\veps(x_\veps) \geq \liminf_{\veps \to
  0}\frac{1}{2\veps} \int_0^1 |\nabla V(x_\veps(s))|^2 \, \dd s \geq
\liminf_{\veps \to 0} \frac{M}{4 \veps} \lambda^2=+\infty.
\end{align*}

Finally we consider the case where $I_0(x)$ is finite.
The proof of the  $\liminf$-inequality follows directly from the definition, as we now show.
Since $x \in BV([0,1])$ and $I_0(x)$ is finite, the set of discontinuity
points $\mathcal D \subset [0,1]$ is finite. We assume next that $\mathcal D \subset
(0,1)$, i.e. $x$ is continuous at the end-points. If $x$ jumps at one or both end points
the same argument can be repeated with obvious modifications of the cut-offs.

Thanks to this assumption for each $\tau \in \mathcal D$ the left-sided and
right-sided limits $x^\pm(\tau)$ exist.
Since $x_\veps$ converges weak-* in $BV([0,1])$ to the piecewise
constant function $x$ Helly's theorem implies that $\lim_{\veps \to 0}
x_\veps(t) = x(t)$ for almost every $t \in [0,1]\setminus \mathcal
D$. Choose
\begin{align*}
\delta = \tfrac{1}{2}\min\left\{\tau-\tau' \; | \; \tau,\tau' \in \mathcal D \cup \{0,1\}, \; \tau < \tau'\right\}.
\end{align*}
and for each $\tau \in \mathcal D$ a pair of such points
$t^\pm\in (0,1)$ such that $\tau-\delta\leq t^-
<\tau < t^+ \leq \tau+\delta$ and $\lim_{\veps \to 0} x_\veps(t^\pm) =
x^\pm(\tau).$ Setting $y_\veps(t) =x_\veps(\veps t)$ one obtains
\begin{align*}
&J_\veps(x_\veps) = \frac{1}{2}\int_0^1\left(\veps \left|\dot
  x(s)\right|^2 +\frac{1}{\veps}\left| \nabla V\right|^2\right)\,
  \dd s  = \frac{1}{2} \int_0^{1/\veps}\left(\left|\dot
  y_\veps(s)\right|^2 +\left| \nabla V(y_\veps(s))\right|^2\right)\, \dd s\\
 \ge &\sum_{\tau \in \mathcal D} \frac{1}{2}\int_{(t^-(\tau)-\delta)/\veps}^{(t^+(\tau)+\delta)/\veps} \left(\left|\dot
  y_\veps(s)\right|^2 +\left| \nabla V(y_\veps(s))\right|^2\right)\, \dd s.
\end{align*}
To justify the boundary condition we modify the path segments $y_\veps(t-\tau/\veps)$ in a way such that the energy changes only slightly.
Define now the function
$$ \bar y^\tau_\veps(t) = \left\{\begin{array}{rl}
x^-(\tau) & \text{ if } s\leq\frac{t^-}{\veps}-1\\
\bigl(\frac{t^-}{\veps}-s\bigr) x^-(\tau) + \bigl(s-
\frac{t^-}{\veps}+1\bigr)x_\veps(t^-)
& \text{ if } \frac{t^-}{\veps} -1< s< \frac{t^-}{\veps}\\
y_\veps(s) & \text{ if } \frac{t^-}{\veps} \leq s \leq \frac{t^+}{\veps},\\
\bigl(s-\frac{t^+}{\veps}\bigr)x^+(\tau)+\bigl(\frac{t^+}{\veps}
+1-s\bigr) x_\veps(t^+)
& \text{ if } \frac{t^+}{\veps}< s < \frac{t^+}{\veps}+1,\\
x^+(\tau) & \text{ if } s\geq \frac{t^+}{\veps}+1.
\end{array} \right.$$
The construction above implies that $\bar y^\tau_\veps \in X(x^-(\tau),x^+(\tau))$ and thus
$$ J(\bar y^\tau_\veps) \geq \Phi(x^-(\tau),x^-(\tau)).$$
Furthermore, the modification of the function $y_\veps$ affects the energy
only in a negligible way, i.e.
\begin{align*}
\sum_{\tau \in \mathcal D}
\frac{1}{2}\int_{(\tau-\delta)/\veps}^{(\tau+\delta)/\veps} \left(\left|\dot
  y_\veps(s)\right|^2 +\left| \nabla V(y_\veps(s))\right|^2\right)\, \dd s
\geq \sum_{\tau \in \mathcal D} J(\bar y^\tau_\veps) + o(1).
\end{align*}
Hence, the liminf-inequality
\begin{align*}
\lim_{\veps \to 0} J_\veps(y_\veps) \geq \sum_{\tau \in \mathcal D} \Phi(x^-(\tau),x^+(\tau))
\end{align*}
holds. The same argument also shows that
\begin{align*}
\liminf_{\veps \to 0} J_\veps(x_\veps) = + \infty,
\end{align*}
if the set $\mathcal D$ is infinite.

To construct a recovery sequence we fix $\delta>0$. By definition for
each $\tau \in \mathcal D$ there exists
$$y^\tau_{\delta} \in X(x^-(\tau),x^+(\tau))$$ such that the support of $\dot y^\tau_{\delta}$ is compact for all
$\tau \in \mathcal D$ and $J(y^\tau_{\delta}) \leq
\Phi(x^-(\tau),x^+(\tau))+\delta$.  It can be checked that for fixed
$\delta>0$ the sequence
$$y_{\veps,\delta}(t) = x(0) + \frac{1}{\veps}\sum_{\tau \in \mathcal D} \int_0^t
\dot y^\tau_\delta((s-\tau)/\veps)\, \dd s $$
converges weak-* in $BV([0,1])$ to $x$ as $\veps \to 0$.

Moreover, there exists a function $\delta(\veps)$ such that
$\lim_{\veps \to 0}\delta(\veps)=o(1)$ and the supports of the functions
$\dot y^\tau_\delta((\cdot -\tau)/\veps)$, $\tau \in \mathcal D$ are disjoint as
$\veps$ tends to 0.
Hence,
\begin{align*}
&J_\veps(y_{\veps,\delta(\veps)}) = \sum_{\tau \in \mathcal D} J_\veps(y^\tau_\delta((\cdot-\tau)/\veps))
\leq \sum_{\tau \in \mathcal D}\Phi(x^-(\tau),x^+(\tau)) + \#\mathcal D \, \delta(\veps)\\
=&\sum_{\tau \in \mathcal D}\Phi(x^-(\tau),x^+(\tau))+o(1)
\end{align*}
as $\veps$ tends to 0.
\end{proof}

\section{Numerical Experiments}
\label{sec:num}

The aim of these numerical experiments is to illustrate
that the $\Gamma-$limit derived in section \ref{sec:gamma}
accurately captures the behaviour
of the problem of minimizing $\Ib$ given by \eqref{eq:pdf}
when $\eps$ is small.
The intuitive picture of the $\Gamma-$limit is that it
is comprised of minimizers with the following properties:
\begin{enumerate}
\item The minimizers are BV functions supported on the
set of critical points of $V$ (see Theorem \ref{t:gamlim}).
\item The contribution to the limit functional from
the jumps in these BV functions (ie from $\Iob$)
can be expressed in terms
of sums of integrals
$\int_{\infty}^{-\infty} |\nabla V(\xs(s))|^2\dd s$
where $\xs$ is a heteroclinic orbit for the Hamiltonian
equations \eqref{eq:el} (see Lemma \ref{lem:lower}, equation
\eqref{eq:one}).
\item Unless both equilibria $x^{\pm}$ are saddle
points, then these heteroclinic
orbits will also be heteroclinic for the gradient
system \eqref{eq:grad} and then the contribution to
the limit functional from $\Iob$
can be expressed in terms of sums
$|V(\xs(\infty))-V(\xs(-\infty))|$
(see Lemma \ref{lem:lower}, equation \eqref{eq:two}).

\item Minimizers $\xs$ will seek
to organize the support of the
limiting BV function so as to minimize the value of
$\int_0^1 \triangle V(x^{\star})(s) \dd s;$
this is the second contribution
to the infimum defined in Theorem \ref{t:gamlim}.
\end{enumerate}

A variety of numerical computations, all of which
exhibit these phenomena on a range of problems,
including high dimensional systems arising
in vacancy diffusion and the Lennard-Jones $38$ cluster,
may be found in the paper \cite{ps09}.
The purpose of this section is to illustrate the four points
above on a single low dimensional example and relate the
results in an explicit way to the theory developed
in earlier sections.
We employ the potential $V:\bbR^2 \to \bbR$ given by
\[
V(x_1,x_2)=
\left(x_1^2+x_2^2\right)
\left((x_1^{\,}-1)^2+x_2^2\right)
\left(x_1^2+(x_2^{\,}-1)^2\right)
\]
which is shown in Figure \ref{fig:example1}.
The potential has three wells of equal depth, situated at $M_0=(0,\,0)$, $M_1=(1,\,0)$,
and $M_2=(0,\,1)$. Saddle points exist at $S_1=\left(  (2+\sqrt{2} )/6, \, (2-\sqrt{2} )/6 \right) $
and (by symmetry) at $S_2=\left(  (2-\sqrt{2} )/6, \, (2+\sqrt{2} )/6 \right) $. The potential is zero at the  minima and attains a value of $2/27$ at the saddles.
The Laplacian of $V$ has value zero at the saddle points,
$4$ at the minimum $M_0$ and $8$ at the minima $M_1$ and $M_2.$

In the following numerical experiments we use gradient
descent to minimize $\Ib$ or $J_{\veps}$ given by
\eqref{eq:Ib} and \eqref{eq:IJ}.
In all the experiments we
employ a value of $\veps^{-1}=10^{-3}$ which proved
to be small enough to exhibit the behaviour of the
$\Gamma-$limit.  We solve the parabolic
PDE arising from the $L^2$ gradient flow for
$\Ib$ (resp. $J_{\veps}$)
by means of a linearly implicit method with
stepsize chosen to ensure decrease of $\Ib$
(resp. $J_{\veps}$) at each time-step.

\begin{figure}[htbp] 
   \centering
\includegraphics[width=4in]{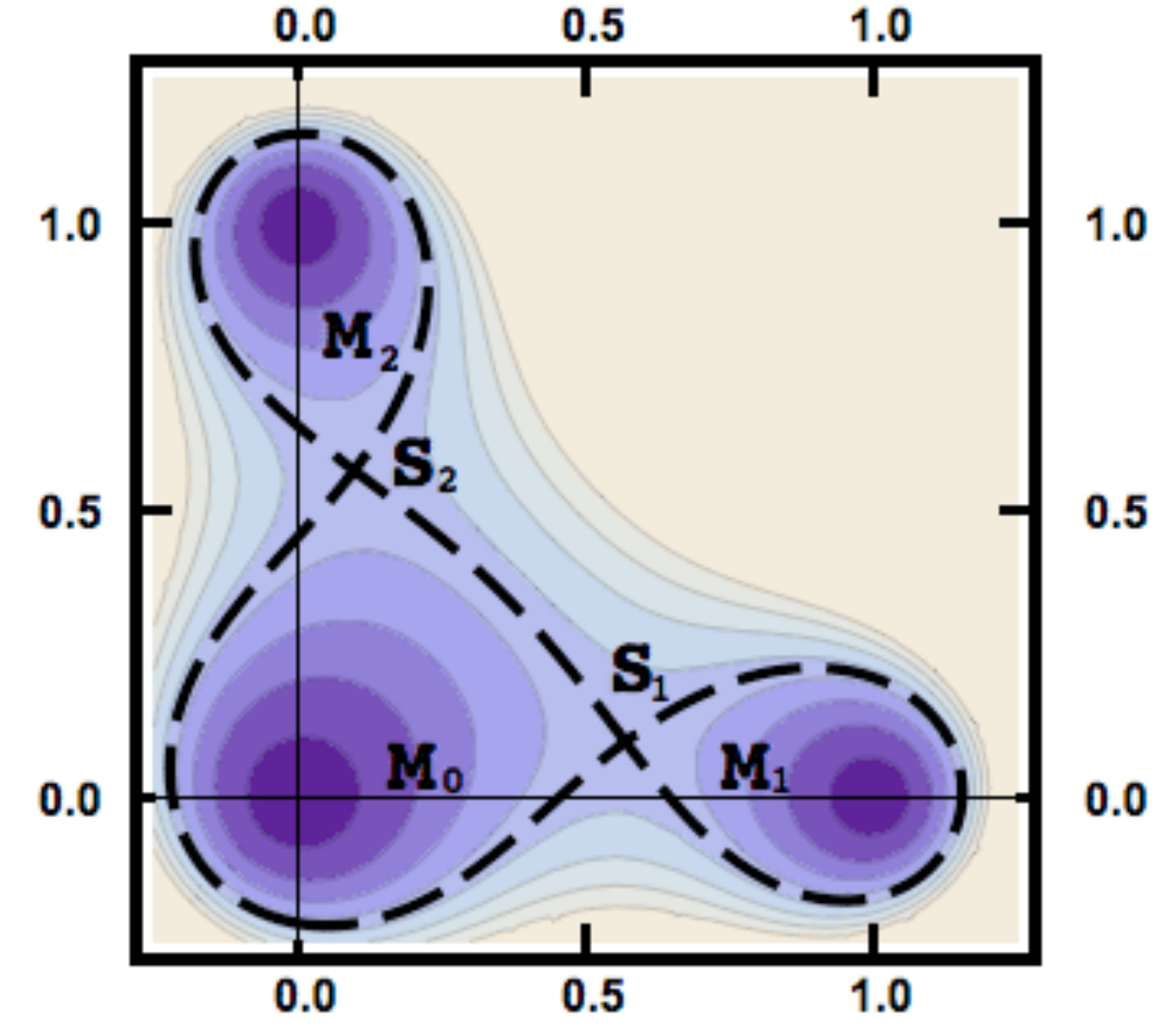}
   \caption{Contour plot of the potential. The dashed black line is
     the equipotential contour that connects the saddle points, $S_1$
     and $S_2$. The three equal-depth minima are labeled: $M_0$,
     $M_1$, and $M_2$.  }
   \label{fig:example1}
\end{figure}

Figure \ref{fig:example2} shows heteroclinic orbits
for the gradient flow \eqref{eq:grad} (in green)
and for the Hamiltonian flow \eqref{eq:el} (in blue).
Figure \ref{fig:example3} shows minimizers of $J_{\veps}$,
both connecting $M_1$ and $M_2$.
It is instructive to compare this figure with the
preceding Figure \ref{fig:example2}. The green
curve in Figure \ref{fig:example3}
connects $M_1$ to $M_2$ via $M_0$ and is
comprised of $4$ segments each of which has been verified
to be approximately given by the gradient heteroclinic orbits
\eqref{eq:grad}.
The blue curve connects $M_1$ to $M_2$ via $S_1$ and $S_2$
and has been verified to be approximately
given by Hamiltonian heteroclinic orbits
satisfying \eqref{eq:el}.
Furthermore, in the gradient case we have verified
that the minimizers obey the sum rule \eqref{eq:two}
This illustrates the connection between minimizers of
$J_{\veps}$ and solutions of the Euler Lagrange
equations for minimizers of $J$, and points 2. and 3. in
particular.

\begin{figure}[htbp] 
   \centering
   \includegraphics[width=4in]{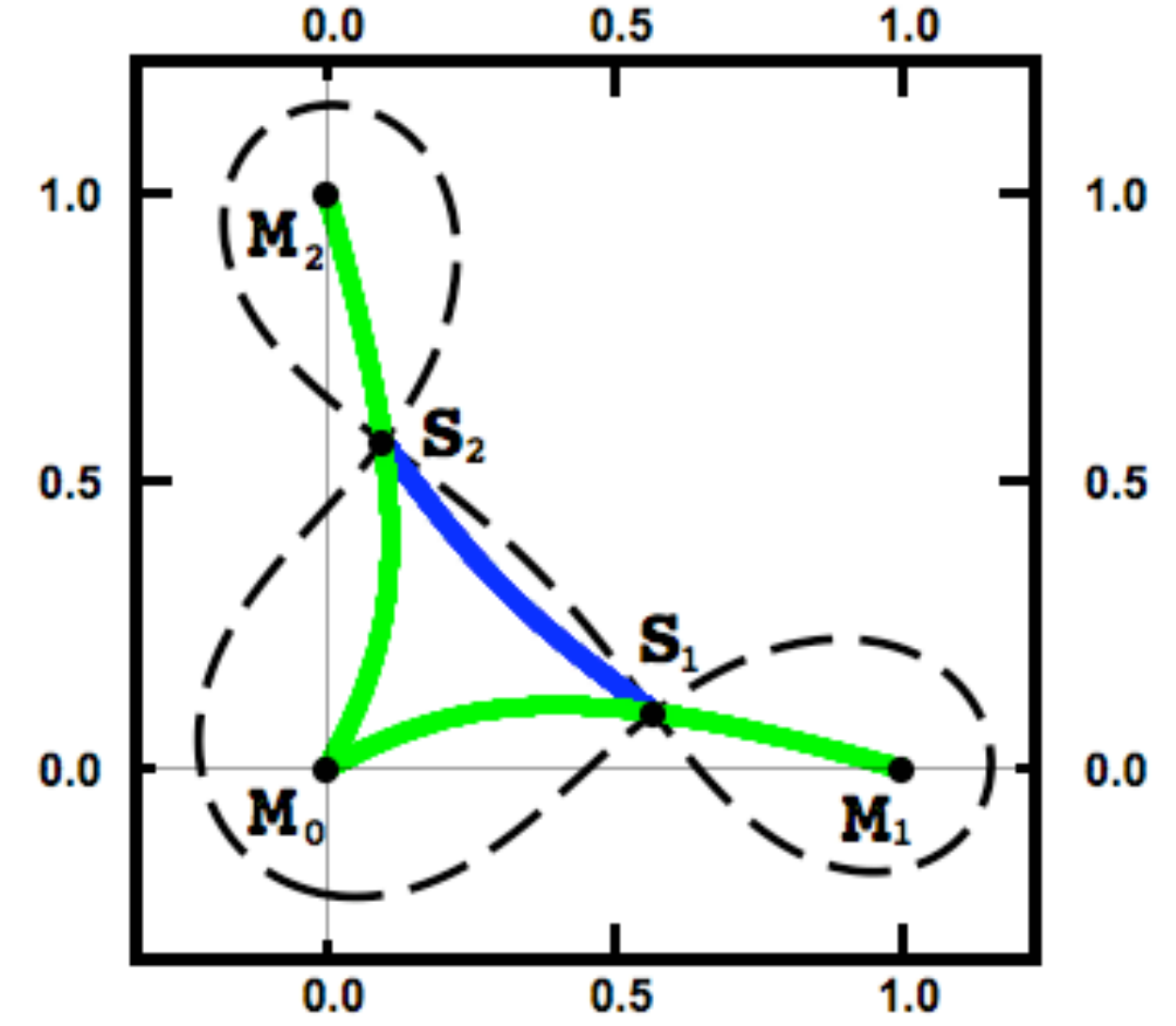}
   \caption{The dashed black line is the equipotential contour that connects the saddle points, $S_1$ and  $S_2$.
   The green line is the path the follows the gradient of the potential from
   $M_1$  to $M_0$ via $S_1$ and then to $M_2$ via $S_2$.
   The blue line is a Hamiltonian minimizer (not a gradient
flow) that  connects $S_1$ to $S_2$.}
   \label{fig:example2}
\end{figure}

\begin{figure}[htbp] 
   \centering
   \includegraphics[width=4in]{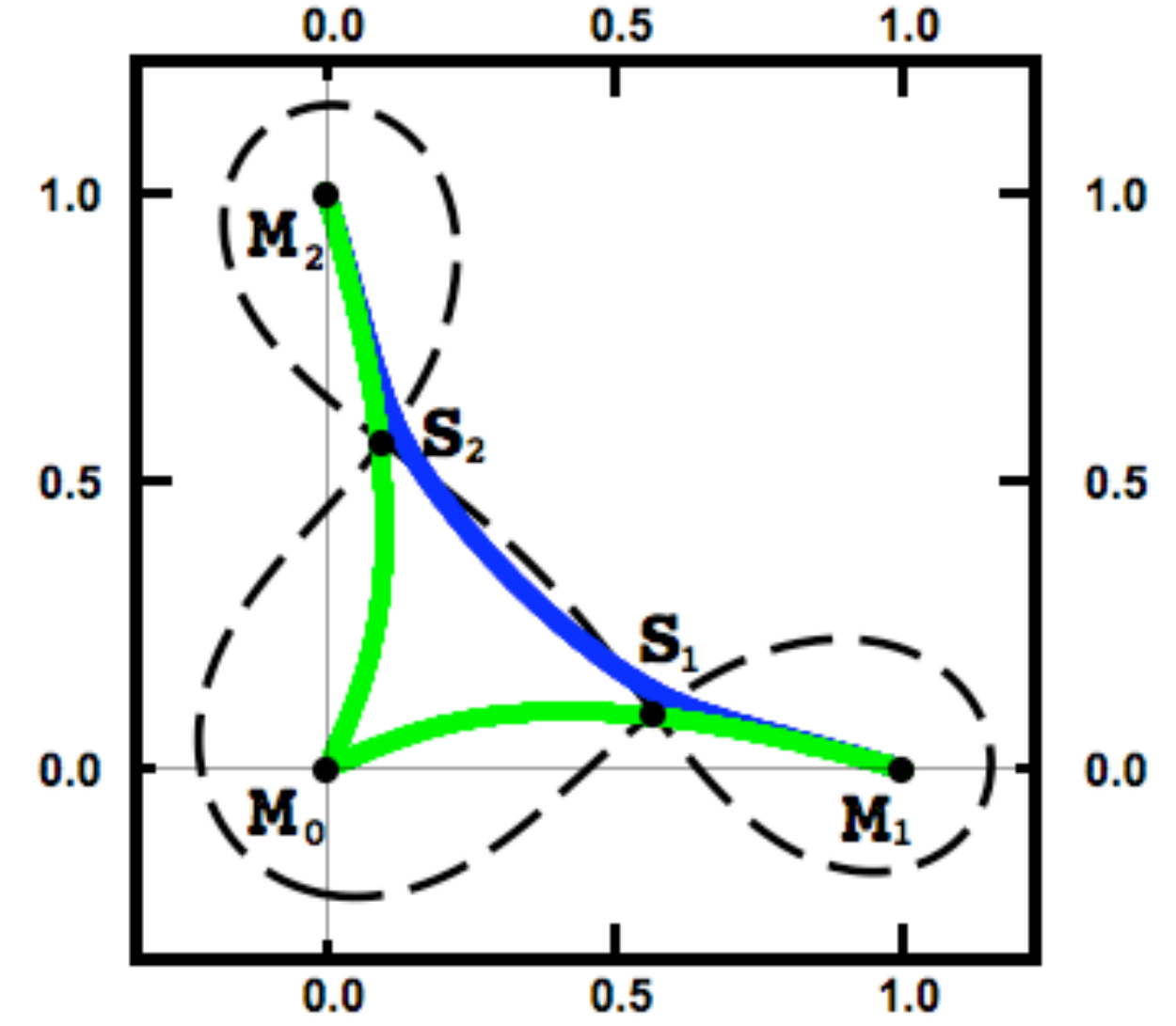}
   \caption{Solutions that minimize $\Iob$ at
$\veps=10^{-3}$.
The dashed black line is the equipotential contour that connects the saddle points, $S_1$ and  $S_2$.
   The green line is the path that starts at
   $M_1$ and proceeds to $M_2$ via $M_0$.
The blue line is the path path that starts at
   $M_1$ and proceeds to $M_2$ avoiding $M_0$. }
 \label{fig:example3}
\end{figure}

Figures \ref{fig:example4} and \ref{fig:example5} show
minimizers of $\Iob$ and $\Ib$ which connect the two
saddle points $S_1$ and $S_2$. Several approximate
 minimizers are shown in each case, found from
different starting points for the gradient flow.
The experiments illustrate
point 1. as they show that the solutions concentrate
on critical points of $V$: in this case simply the
two saddles. They also illustrate point 4. as they show
that, in this case, the minimizers of $\Iob$ and $\Ib$ are
indistinguishable; this is because the Laplacian of $V$
is zero at the saddle points.

Figures \ref{fig:example6} and \ref{fig:example7} also show
minimizers of $\Iob$ and $\Ib$ which connect the two
saddle points $S_1$ and $S_2$. However the starting
points for the gradient flow differ from
those used to generate Figures \ref{fig:example4}
and \ref{fig:example5}; in particular they
are based on
a function which passes through the minimum $M_0$. As
a consequence the minimizers also pass through $M_0.$
For $\Iob$ there are then multiple approximate minimizers,
all supported on $S_1, S_2$ and $M_0$.
However the support
can be organized more or less arbitrarily (provided only
two transitions occur) to obtain approximately the
same value of $\Iob$; we show a solution where the
support is organized symmetrically. The situation for $\Ib$ is quite
different: the effect of the Laplacian of $V$, which is $4$
at $M_0$ and $0$ at $S_1$ and $S_2$, means that minimizers
place most of their support at $M_0.$  The experiments
thus again illustrate point 1. as they show that the
solutions concentrate
on critical points of $V$. They also illustrate point 4.

\begin{figure}[htbp] 
   \centering
   \includegraphics[width=4in]{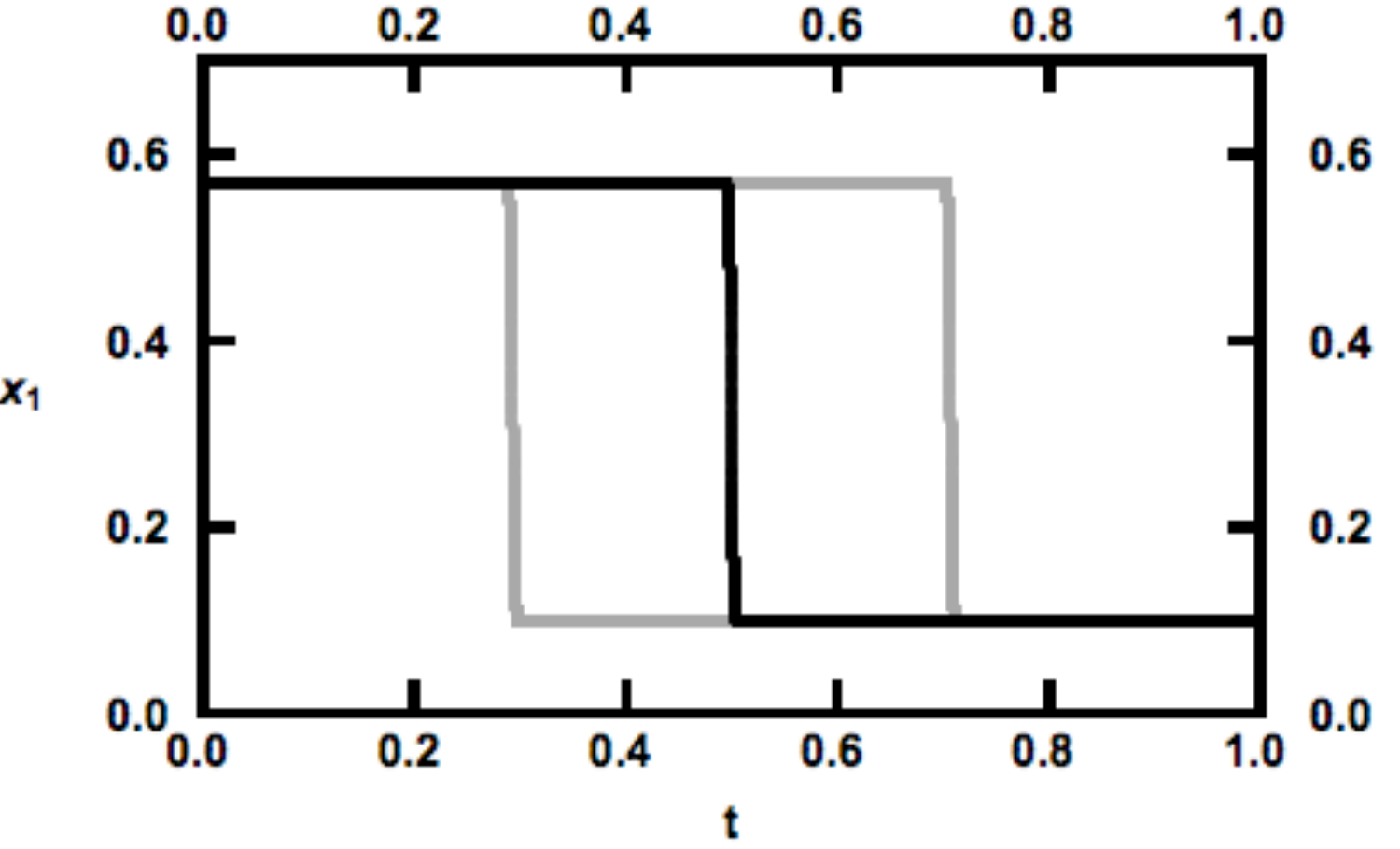}
   \caption{The $x_1$ component of the paths that minimize $\Iob$ at
     $\veps=10^{-3}$. The paths start at $S_1$ and are conditioned to
     end at $S_2$.  The ending path depends on the nature of the
     starting path. The results of using three different starting
     paths are displayed.  None of the initial paths approach the
     origin and thus avoid passing through $M_0$. The black curve
     corresponds to a symmetric minimizer.  }
 \label{fig:example4}
\end{figure}

\begin{figure}[htbp] 
   \centering
   \includegraphics[width=4in]{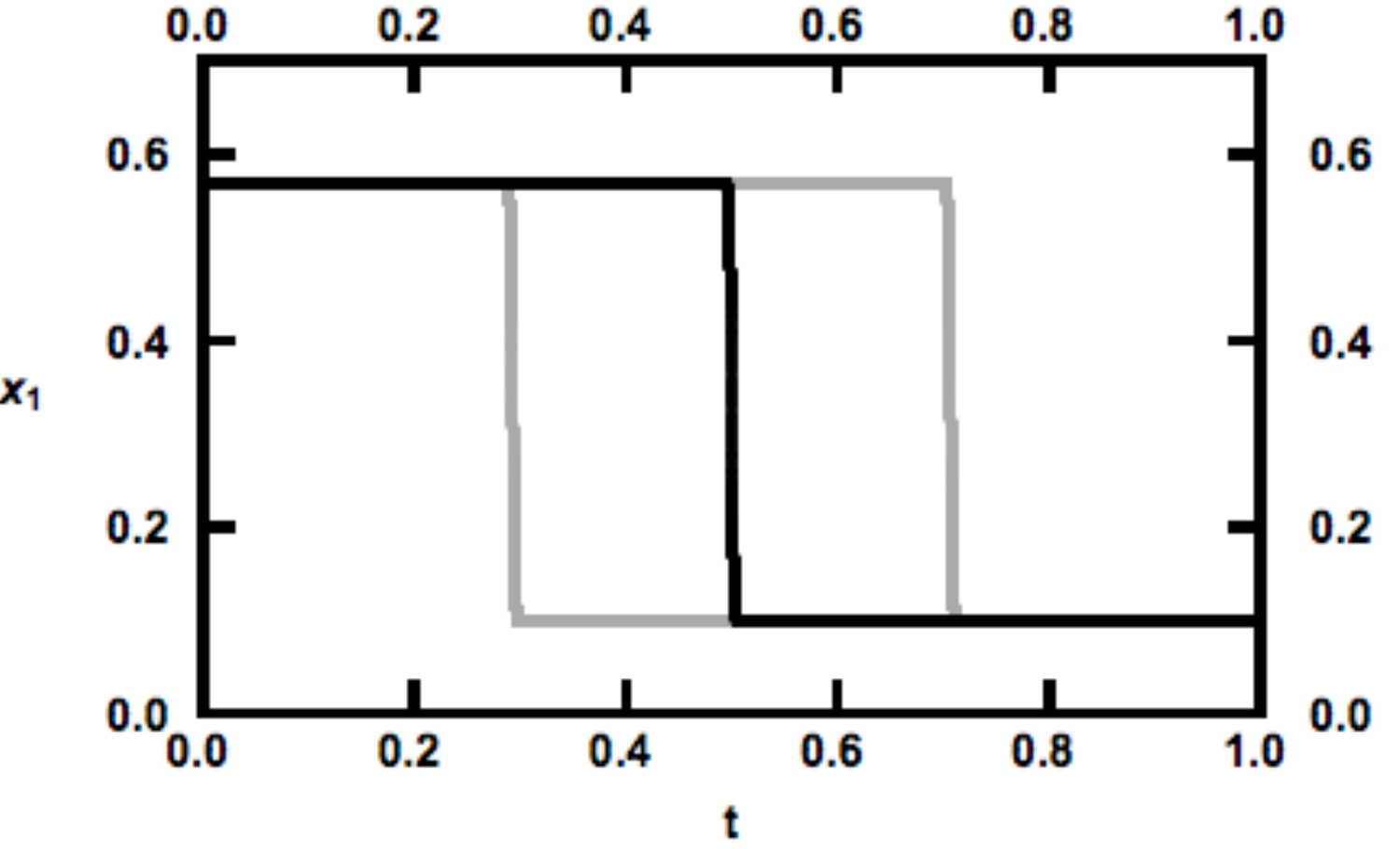}
   \caption{The $x_1$ component of the paths that minimize $\Ib$ at
     $\veps=10^{-3}$.  The paths start at $S_1$ and are conditioned to
     end at $S_2$.  The ending path depends on the nature of the
     starting path. The results of using three different starting
     paths are displayed.  None of the initial paths approach the
     origin and thus avoid passing through $M_0$. The black curve
     corresponds to a symmetric minimizer.  Note that this plot is
     indistinguishable from the previous one.  }
 \label{fig:example5}
\end{figure}

\begin{figure}[htbp] 
   \centering
   \includegraphics[width=4in]{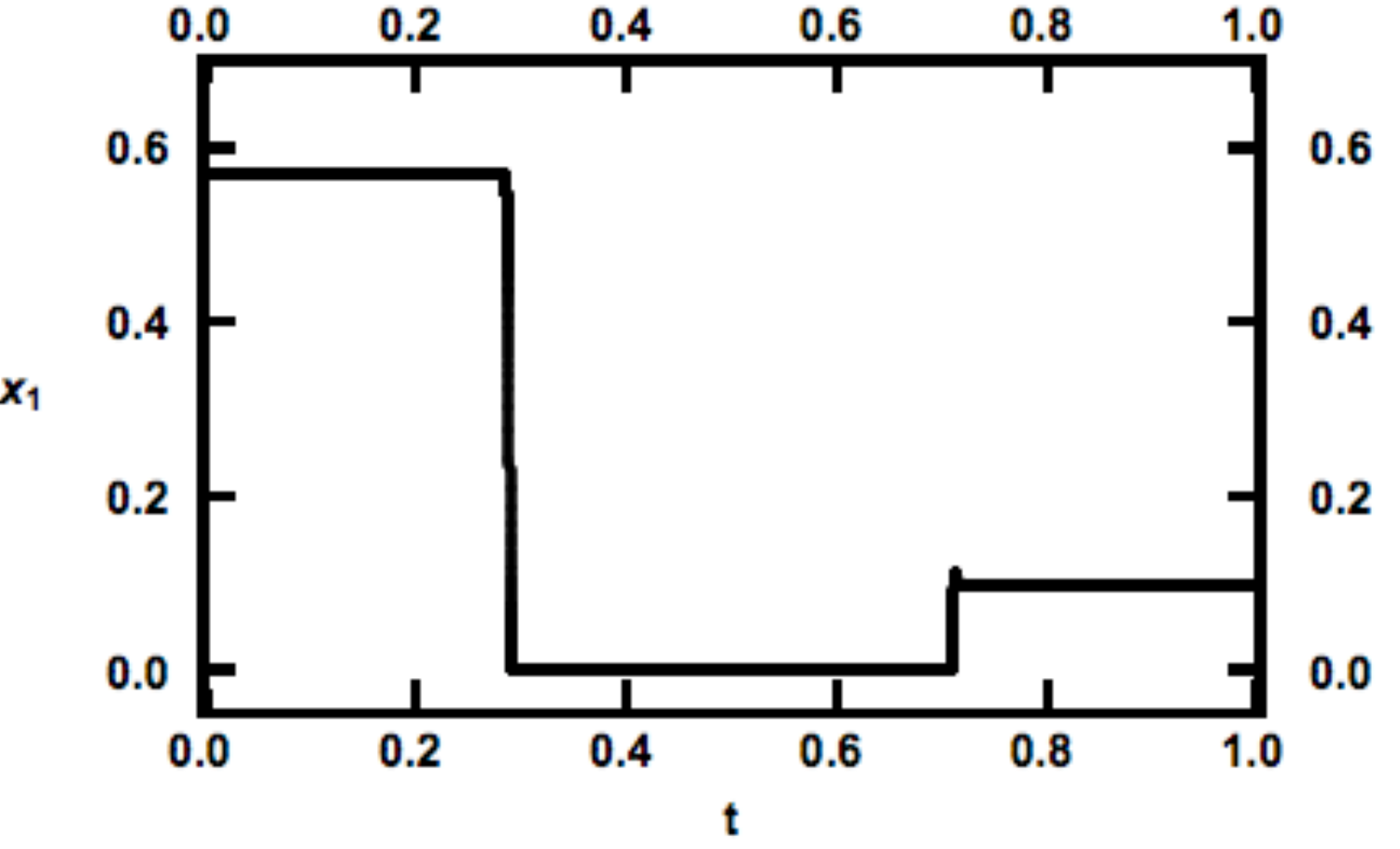}
 \caption{The $x_1$ component of the path that minimize $\Iob$ at
$\veps=10^{-3}$. The path starts at $S_1$ and is conditioned to end at $S_2$.  The initial path contains the origin $M_0$
and thus the ending path spends an arbitrary fraction of the time at $M_0$.  The ending path depends on the nature of the starting path, in particular where the path crosses the origin.  Only the symmetric minimizer is displayed here. }
 \label{fig:example6}
\end{figure}

\begin{figure}[htbp] 
   \centering
   \includegraphics[width=4in]{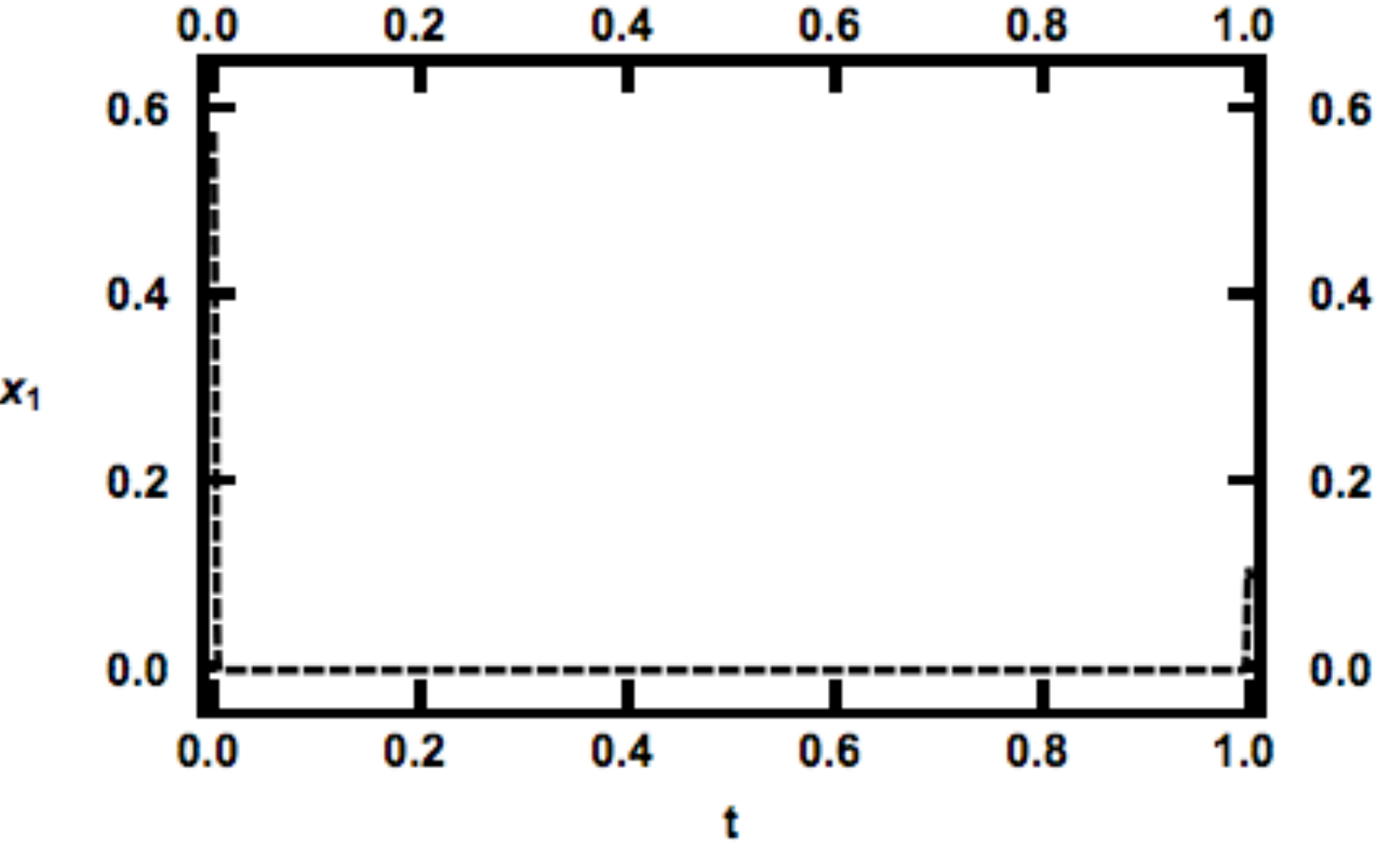}
   \caption{The $x_1$ component of the path that minimize $\Ib$
at $\veps=10^{-3}$. The path starts at $S_1$ and is conditioned to end at $S_2$.  The initial path contains the origin $M_0$
and thus the ending path is dominated by $M_0$.   }
 \label{fig:example7}
\end{figure}

Figures \ref{fig:example8} and \ref{fig:example9} again show
minimizers of $\Iob$ and $\Ib$, now connecting the two
minima $M_1$ and $M_2$, and constructed to pass through
the other minimum $M_0$ and the two saddle points $S_1$
and $S_2$.
For $\Iob$ there are then multiple approximate minimizers,
all supported on the five critical points, one
of which is shown in Figure \ref{fig:example8},
a solution where the support is organized symmetrically.
The situation for $\Ib$ is again very different:
the effect of the Laplacian of $V$, which is $8$
at $M_1$ and $M_2$, means that minimizers
place most of their support at these two points,
as shown in Figure \ref{fig:example9}. The single
interface in fact contains several transitions,
and hence several contributions to the $\Gamma-$limit.
Furthermore this single interface can be placed
arbitrarily; we have shown a symmetric case.
The experiments once again illustrate points 1. and 4.

\begin{figure}[htbp] 
   \centering
   \includegraphics[width=4in]{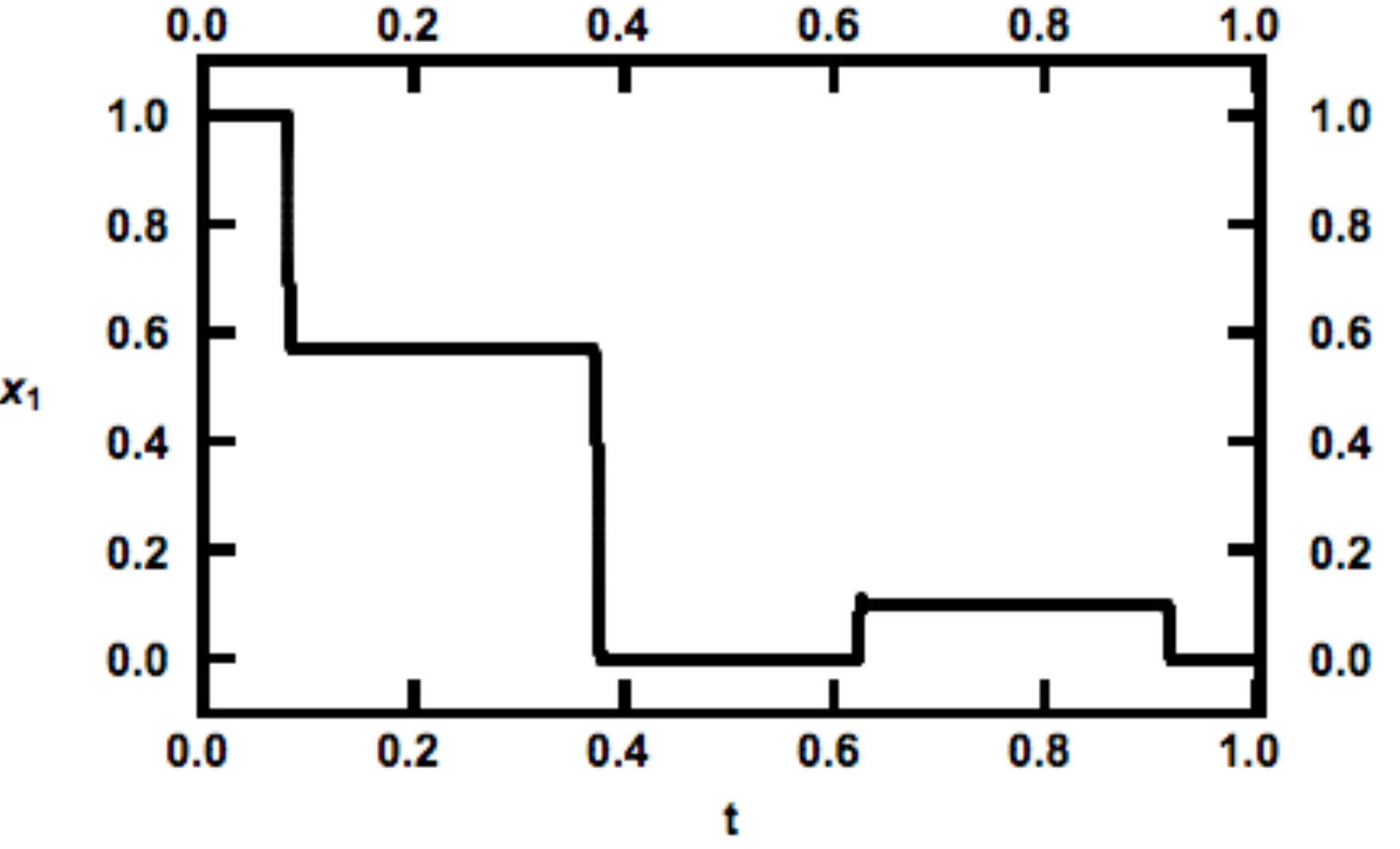}
 \caption{The $x_1$ component of the path that minimize $\Iob$ at $\veps=10^{-3}.$
The path starts at $M_1$ and is conditioned to end at $M_2$.  The path breaks into several segments. The initial path contains the origin and thus the "ending" path passes through  $M_0$.
 The ending path depends on the nature of the starting path, in particular where the path crosses the origin.  Only the symmetric minimizer is displayed here. }
 \label{fig:example8}
\end{figure}

\begin{figure}[htbp] 
   \centering
   \includegraphics[width=4in]{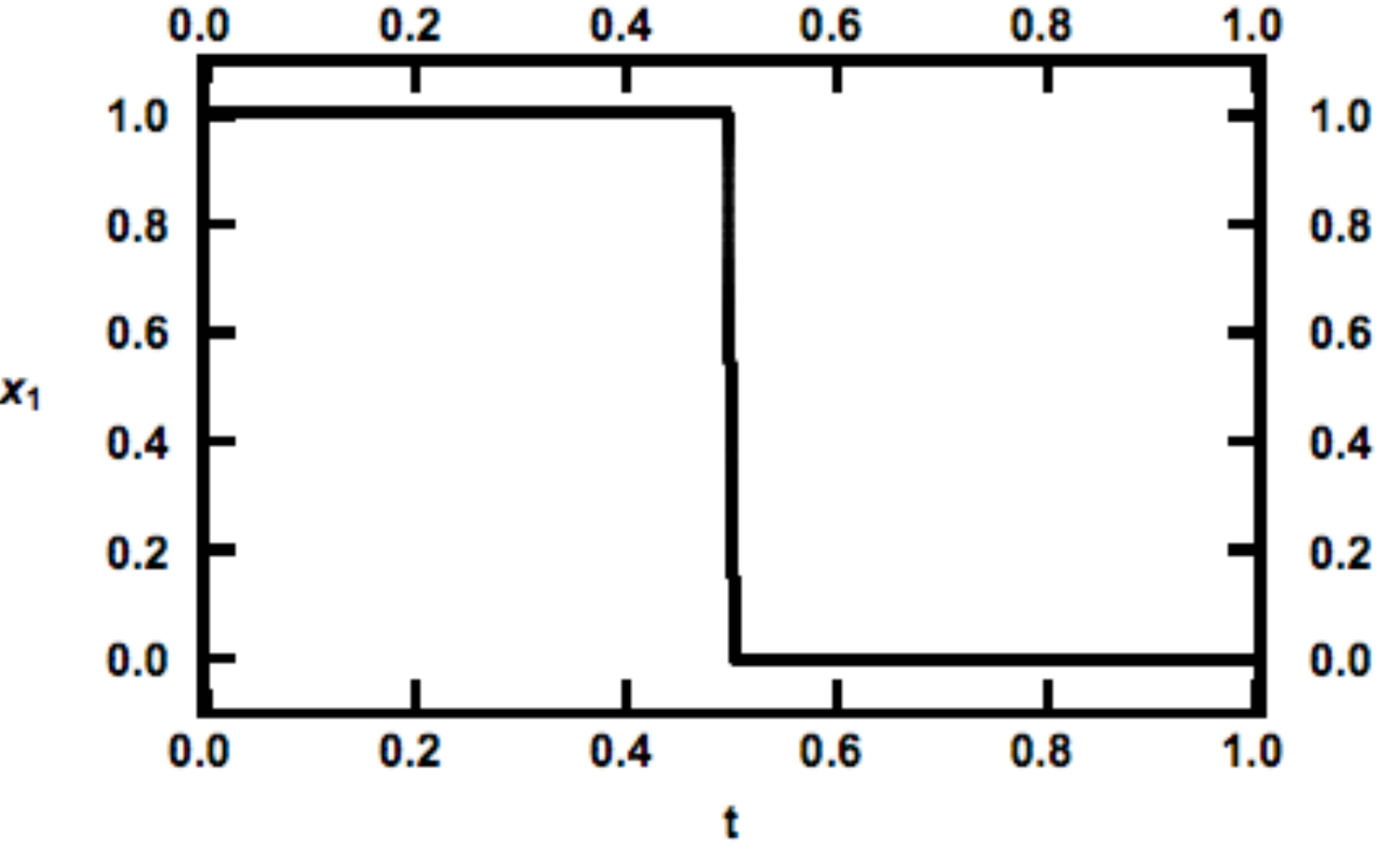}
 \caption{The $x_1$ component of the path that minimize $\Ib$
at $\beta=1000$. The path starts at $M_1$ and is conditioned to end at $M_2$.  The initial path does not approach the origin $M_0$
and thus the ending path also avoids $M_0$.   Only the symmetric minimizer is displayed here. Note that all the "activity" is
 scrunched into a small time interval.}
 \label{fig:example9}
\end{figure}

\vspace{0.2in}

\noindent{\bf Acknowledgements.} The authors are grateful
to Eric Vanden Eijnden for helpful disucssions. AMS is grateful
to EPSRC and ERC for financial support.

\bibliography{gamma}
\bibliographystyle{unsrt}

\end{document}